\documentclass[11pt,reqno]{amsart}

\usepackage[arrow,matrix,curve]{xy}
\usepackage{epigraph}

\usepackage[dvips]{graphicx} 
\usepackage{mathtools}

\usepackage{amssymb, latexsym, amsmath, amscd, array, hyperref, esint
}

\usepackage{breakurl}   

\makeatletter
\@namedef{subjclassname@2020}{\textup{2020} Mathematics Subject Classification}
\makeatother

\newtheorem{theorem}{Theorem}[section]
\newtheorem{lemma}[theorem]{Lemma}

\newtheorem{proposition}[theorem]{Proposition}
\newtheorem{corollary}[theorem]{Corollary}

\theoremstyle{definition}
\newtheorem{definition}[theorem]{Definition}

\newtheorem{remark}[theorem]{Remark}

\numberwithin{equation}{section}

\newcommand\R {{\mathbb R}}

\newcommand\T {{\mathbb T}} 
\newcommand\K {{\mathbb K}}

\DeclareMathOperator{\area}{{\rm area}} 
\DeclareMathOperator{\vol}{{\rm vol}} 
 
\DeclareMathOperator{\sys}{{\rm sys}}

\DeclareMathOperator{\arcsinh}{arcsinh}

\newcommand{\RP}{{\mathbb R\mathbb P}} 
 
\newcommand{\KK}{K}

\long\def\forget#1\forgotten{} %

\numberwithin{equation}{section}

\title{Nonpositively curved surfaces are Loewner}

\author[M. Katz]{Mikhail G. Katz} \address{Department of
  Mathematics, Bar Ilan University, Ramat Gan 5290002 Israel}
\email{katzmik@math.biu.ac.il}

\author[S.\;Sabourau]{St\'ephane Sabourau}
\address{\parbox{\linewidth}{Univ Paris Est Creteil, CNRS, LAMA, F-94010 Creteil, France \\
Univ Gustave Eiffel, LAMA, F-77447 Marne-la-Vall\'ee, France}}
\email{stephane.sabourau@u-pec.fr}


\begin{document}

\maketitle

\thispagestyle{empty}

{\centering
\emph{Dedicated to Eugenio Calabi's memory.}
\par}

\begin{abstract}
We show that every closed nonpositively curved surface satisfies
Loewner's systolic inequality.
The proof relies on a combination of the Gauss--Bonnet formula with an
averaging argument using the invariance of the Liouville measure under
the geodesic flow.  This enables us to find a disk with large total
curvature around its center yielding a large area.
\end{abstract}






\section{Introduction}

The systole of a closed nonsimply connected surface~$M$ endowed with a Riemannian metric, denoted by~$\sys(M)$, is defined as the length of the shortest noncontractible loop of~$M$.
It is attained by the length of a noncontractible closed geodesic.
The \emph{systolic area} of~$M$ is defined as
\[
\sigma(M) = \frac{\area(M)}{\sys(M)^2}.
\]
We will say that $M$ is \emph{Loewner} if its systolic area satisfies
\[
\sigma(M) \geq \frac{\sqrt{3}}{2} \approx 0.866.
\]
The first systolic inequality, due to Loewner, asserts that every metric on the torus~$\T^2$ is Loewner; see~\cite{pu}.
By~\cite{KS06pams}, all surfaces with a Riemannian metric in a hyperelliptic conformal class are Loewner.
In particular, every genus two surface is Loewner since every conformal class in genus two is hyperelliptic.
By~\cite{KS05}, surfaces of genus at least~$20$ are Loewner.
Pushing this technique further, Li and Su announce that this still holds true for  surfaces of genus~$18$ and $19$, and even for surfaces of genus~$\geq 11$ when one restricts to nonpositively curved metrics; see~\cite{LS}.
For the other cases
the problem is still open.

In this article, we resolve this problem for nonpositively curved surfaces of any genus (as well as for nonorientable surfaces) relying on a new approach.

\begin{theorem} \label{theo:main}
Every closed nonpositively curved surface is Loewner.
\end{theorem}

Actually, we prove a stronger statement; see Proposition~\ref{prop:end}.
Let $M$ be a closed nonpositively curved surface.
For every $r \in [0,\frac{1}{2} \sys(M)]$, there is a ball~$B(r) \subseteq M$ of radius~$r$ with
\[
\area B(r) \geq \pi r^2 - \frac{2\pi \chi(M)}{\area(M)} \cdot \frac{\pi}{12} r^4.
\]
In particular, for $r=\frac{1}{2} \sys(M)$, we have
\[
\area B(\tfrac{1}{2} \sys(M)) \geq \left( \frac{\pi}{4} - \frac{\pi^2 \chi(M)}{96} \cdot \frac{\sys(M)^2}{\area(M)} \right) \sys(M)^2.
\]
Now, if $M$ is not a torus, we can show (see Corollary~\ref{coro}) that its systolic area is at least 
\[
\frac{\pi}{8} \left( 1 + \sqrt{\tfrac{7}{3}} \right) \approx 0.992
\]
which is enough to conclude. \\


Our strategy differs from previous works.
Here is the rough idea.
Suppose that $M$ is a genus~$g$ surface which is not Loewner.
Normalize the metric so that $\sys(M)=1$.
Since the metric is nonpositively curved, every disk of radius~$\frac{1}{4}$ has area at least~$\frac{\pi}{16}$, which represents at least~$\frac{\pi}{16} \slash \frac{\sqrt{3}}{2} \approx 22.6\%$ of the total area.
Now, by an averaging argument using the invariance of the Liouville measure under the geodesic flow, one should be able to find a disk~$D$ of radius~$\frac{1}{4}$ with at least~$22.6\%$ of the total curvature. 
The disk~$D_+$ of radius~$\frac{1}{2}$ centered at the same point has a lot of (negative) curvature around its center, namely in~$D$.
This should force the disk~$D_+$ to have a lot of area contradicting the $\frac{\sqrt{3}}{2}$-bound on the surface area.
At implementation level, the existence of a curvature-rich disk relies on an integral-geometric formula relating the weighted average of the curvature~$K$ on tangent disks in~$TM$ with the Euler characteristic and a comparison result between this weighted average and another weighted average of~$K$ on metric disks in~$M$.
The area lower bound on this curvature-rich disk follows from an expression relating the area of this disk with the previous weighted average of~$K$ on the same disk.

Recent progress in systolic geometry includes \cite{Go24} and
\cite{Ka24a}.

\forget
Our strategy differs from previous works.
Here is the rough idea.
Suppose that $M$ is a genus~$g$ surface which is not Loewner.
Normalize the metric so that $\sys(M)=1$.
Since the metric is nonpositively curved, every disk of radius~$\frac{1}{4}$ has area at least~$\frac{\pi}{16}$, which represents at least~$\frac{\pi}{16} \slash \frac{\sqrt{3}}{2} \approx 22.6\%$ of the total area.
Now, by an averaging argument using the invariance of the Liouville measure under the geodesic flow, one should be able to find a disk~$D$ of radius~$\frac{1}{4}$ with at least~$22.6\%$ of the total curvature, that is $0.226 \times 4\pi(g-1)$.
The disk~$D_+$ of radius~$\frac{1}{2}$ centered at the same point has a lot of (negative) curvature around its center, namely in~$D$.
This should force the disk~$D_+$ to have a lot of area.
Specifically, it should have more area than the flat disk of radius~$\frac{1}{2}$ with a conical singularity at distance~$\frac{1}{4}$ from its center concentrating all the curvature.
This leads to a contradiction since the area of this extremal disk is equal to
\[
\frac{\pi}{4} + 0.226 \times 4\pi(g-1) \times \frac{1}{2} \times \frac{1}{16} \geq 0.874 > \frac{\sqrt{3}}{2}.
\]
In the proof of the main theorem, we build around this idea making it more precise and getting better estimates.
\forgotten

\section{Disks and curvature}

Let $M$ be a closed surface of nonpositive Euler characteristic~$\chi(M)$ endowed with a Riemannian metric.
Let $\pi:UM \to M$ be the canonical projection defined on the unit tangent bundle~$UM$ of~$M$.
We will sometimes denote a unit tangent vector~$u \in UM$ by $u_x \in U_xM$  when we want to emphasize its basepoint~$x =\pi(u)$ in~$M$.

The Liouville measure on~$UM$ decomposes as 
\begin{equation} \label{eq:decomp}
du = dx \, du_x
\end{equation}
where $dx$ is the area measure of~$M$ and $du_x$ is the canonical length measure of~$U_xM$; see~\cite[\S1.M]{besse}.
Note that the Liouville measure is invariant under the geodesic flow~$\varphi_t:UM \to UM$ of~$M$ and that
\[
\vol(UM) = 2 \pi \area(M).
\]

Let $\KK$ be the Gaussian curvature of~$M$.
The Gauss--Bonnet formula for a domain~$D$ of~$M$ with piecewise smooth boundary~$\partial D$ can be written
\begin{equation} \label{eq:GBB}
\int_D \KK(x) \, dx + \int_{\partial D} \kappa(s) \, ds + \tau_{\partial D} = 2\pi \chi(D)
\end{equation}
where $\chi(D)$ is the Euler characteristic of~$D$, $\kappa$ is the geodesic curvature of~$\partial D$ and~$\tau_{\partial D}$ is the sum of the angular differences of the tangent vectors at the corner points of~$\partial D$.
When $D=M$, the Gauss--Bonnet formula for~$M$ takes the form
\begin{equation} \label{eq:GB}
\int_M \KK(x) \, dx = 2\pi \chi(M).
\end{equation}
It will be convenient to introduce the function
\[
\overline{K}:UM \to \R
\]
defined as $\overline{K}=K \circ \pi$. \\

The results of this section can be summarized as follows.
First, there is an integral-geometric formula relating a weighted average of $\overline{K}$ of certain horizontal disks of radii at most~$r$ in $UM$ in terms of the Euler characteristic and~$r$.  
Furthermore, the curvature condition yields an inequality between this weighted average and the average of~$K$ of certain disks in~$M$.  
The inequality between the two averages entails the existence of a curvature-rich disk. \\

The following function will play a key role in our approach.

\begin{definition}
Let $r>0$.
Define $F_r:M \to \R$ as
\[
F_r(x) = \int_0^r (r-\rho) \int_{U_xM} \int_0^\rho \overline{K}(\varphi_t(u_x))) \, t \, dt \, du_x \, d\rho.
\]
\end{definition}

Let us compute the integral of~$F_r$.

\begin{lemma} \label{lem:F}
We have
\[
\int_{M} F_r(x) \, dx = 2\pi \chi(M) \cdot \frac{\pi r^4}{12}.
\]
\end{lemma}

\begin{proof}
By the Liouville measure decomposition~\eqref{eq:decomp}, we derive
\begin{align*}
\int_{M} F_r(x) \, dx & = \int_{M} \int_0^r (r-\rho) \int_{U_xM} \int_0^\rho \overline{K}(\varphi_t(u_x))) \, t \, dt \, du_x \, d\rho \, dx \\
 & = \int_0^r (r-\rho) \int_0^\rho \int_{UM} \overline{K}(\varphi_t(u))) \, du \, t \, dt \, d\rho.
\end{align*}   
Now, by the invariance of the Liouville measure under the geodesic flow, we obtain
\begin{align*}
\int_{M} F_r(x) \, dx & = \int_0^r (r-\rho) \int_0^\rho \int_{UM} \overline{K}(u)) \, du \, t \, dt \, d\rho \\ 
 & = 2 \pi \, \int_{M} \KK(x) \, dx \, \int_0^r (r-\rho) \int_0^\rho t \, dt \, d\rho \\  
 & = 2\pi \chi(M) \cdot \frac{\pi r^4}{12}
\end{align*}   
where the last equality follows from the Gauss--Bonnet formula~\eqref{eq:GB}.
\end{proof}

Let us introduce the following function in connection with the area lower bound in Proposition~\ref{prop:theta}.

\begin{definition}
Let $r \in (0,\frac{1}{2} \sys(M))$.
Define $G_r:M \to \R$ as
\[
G_r(x) = \int_0^r (r-\rho) \int_{B_x(\rho)} \KK(y) \, dy \, d\rho
\]
where $B_x(\rho)$ is the ball of radius~$\rho$ centered at~$x$.
\end{definition}

The functions~$F_r$ and~$G_r$ are related through the following comparison result.

\begin{lemma} \label{lem:FB}
Suppose $M$ is nonpositively curved.  Let
$r\in(0,\frac{1}{2}\sys(M))$.  Then, for every $x \in M$,
\[
F_r(x) \geq G_r(x).
\]
\end{lemma}

\begin{proof}
By Gauss' Lemma and since $r < \frac{1}{2} \sys(M)$, the exponential map at~$x$ in polar coordinates induces a diffeomorphism 
\[
\begin{aligned}
U_xM \times (0,r] & \longrightarrow B_x(r) \setminus \{x\} \\
(u_x,t) & \longmapsto y=\pi(\varphi_t(u_x))
\end{aligned}
\]
Since $M$ is nonpositively curved, this map is distance nondecreasing by Rauch's comparison theorem, see~\cite[\S1.11]{CE}.
It follows that
\[
t \, dt \, du_x \leq dy.
\]
After integration and since~$\KK \leq 0$, this implies
\[
\int_0^r (r-\rho) \int_{U_xM} \int_0^\rho \overline{K}(\varphi_t(u_x))) \, t \, dt \, du_x \, d\rho \geq \int_0^r (r-\rho) \int_{B_x(\rho)} \KK(y) \, dy \, d\rho.
\]
That is, $F_r(x) \geq G_r(x)$.
\end{proof}

We can now derive our key estimate.

\begin{proposition} \label{prop:B}
Assume $M$ is nonpositively curved.  Let $r \in (0,\frac{1}{2}
\sys(M))$.  Then there exists $x_0 \in M$ such that
\[
G_r(x_0) \leq \frac{2\pi \chi(M)}{\area(M)} \cdot \frac{\pi r^4}{12}.
\]
\end{proposition}

\begin{proof}
Taking the average integral $\fint_{M}$ over~$M$ in Lemma~\ref{lem:FB} leads to
\[
\fint_{M} G_r(x) \, dx \leq \fint_{M} F_r(x) \, dx = \frac{1}{\area(M)} \, \int_{M} F_r(u) \, du.
\]
By Lemma~\ref{lem:F}, this yields
\[
\fint_{M} G_r(x) \, dx \leq \frac{2\pi \chi(M)}{\area(M)} \cdot \frac{\pi r^4}{12}
\]
and the result immediately follows.
\end{proof}

\section{Disks and area}

\forget
\begin{proposition} \label{prop:theta}
Assume $M$ is nonpositively curved.  Denote by
\[
\theta = \int_{B_x(r)} \KK_-(y) \, dy
\]
the opposite of the total curvature of~$B_r(x)$.
Then, for $R \in [r,\frac{1}{2} \sys(M)]$, we have
\[
\area B_x(R) \geq \pi R^2 + \frac{\theta}{2} (R-r)^2.
\] 
\end{proposition}
\forgotten

In this section, we show that disks with large (negative) curvature have a large area, and proceed to the proof of the main theorem. \\

We will need the following area lower bound.

\begin{proposition} \label{prop:theta}
Let $r \in (0,\frac{1}{2} \sys(M))$.
Then
\begin{equation} \label{eq:lower}
\area B_x(r) \geq \pi r^2 - G_r(x).
\end{equation}
\end{proposition}

\begin{proof}
First, approximate the metric on~$M$ by a real analytic metric.
For this new metric, the component~$\mathcal{C}_x(s)$ of the circle of radius~$s \leq r$ centered at~$x$ surrounding its center is a piecewise smooth curve.
Moreover, the length function~$s \mapsto L(\mathcal{C}_x(s))$ is differentiable except for a finite number of values of~$s$, and its derivative is given by the first variation formula.
Specifically, as long as~$B_x(s)$ is nonempty, we have
\[
L'(\mathcal{C}_x(s)) = \int_{\mathcal{C}_x(s)} \kappa(t) \, dt + \tau_s
\]
for almost every~$s$, where $\kappa$ is the geodesic curvature of the curve~$\mathcal{C}_x(s)$ and $\tau_s$ is the sum of the angular difference of the tangent vectors at the corner points of~$\mathcal{C}_x(s)$.
By the Gauss--Bonnet formula~\eqref{eq:GBB} for domains with boundary and since $\mathcal{C}_x(s)$ bounds a topological disk~$\mathcal{D}_x(s)$, we derive
\[
L'(\mathcal{C}_x(s)) = 2 \pi - \int_{\mathcal{D}_x(s)} K(y) \, dy.
\]
Integrating this relation twice and using the coarea formula lead to
\begin{align}
\area \mathcal{D}_x(r) & \geq \pi r^2 - \int_0^r \int_0^\rho  \int_{\mathcal{D}_x(s)} K(y) \, dy \, ds \, d\rho \label{eq:coarea} \\
 & \geq \pi r^2 - \int_0^r (r-\rho) \, \int_{\mathcal{D}_x(\rho)} K(y) \, dy \, d\rho. \nonumber
\end{align}
To conclude, simply observe that, when the real analytic metric approaches the initial metric on~$M$, the domains~$\mathcal{D}_x(\rho)$ Hausdorff converge to the balls~$B_x(\rho)$ (uniformly in~$\rho$).
In particular, the area of~$\mathcal{D}_x(r)$ converges to the one of~$B_x(r)$, and $\int_{\mathcal{D}_x(\rho)} K(y) \, dy$ uniformly converges to $\int_{B_x(\rho)} K(y) \, dy$.
Hence,
\[
\area B_x(r) \geq \pi r^2 - \int_0^r (r-\rho) \, \int_{B_x(\rho)} K(y) \, dy \, d\rho.
\]
\end{proof}

\begin{remark}
When the surface~$M$ is nonpositively curved, there is an equality in~\eqref{eq:lower}.
Indeed, in this case, we can approximate the metric by a real analytic metric of negative curvature.
In this case, the curve~$\mathcal{C}_x(s)$ represents the circle of radius~$s$ around~$x$ and the domain~$\mathcal{D}_x(s)$ it surrounds coincides with the metric ball of radius~$s$ around~$x$.
Because of the curvature condition, this ball is a topological disk.
In this case, the coarea formula leads to an equality in~\eqref{eq:coarea} and in the following inequalities.
\end{remark}

\forget

\begin{proof}
We argue as in the proof of Theorem~2.4 in~\cite{KS21} making the necessary changes.
By metric approximation (see\cite[\S2]{KS21}), we can assume that the metric on~$M$ is piecewise flat with conical singularities.
Under this assumption, the circle~$\mathcal{C}_x(s)$ of radius~$s$ centered at~$x$ is a piecewise smooth curve.
Moreover, the length function~$s \mapsto L(\mathcal{C}_x(s))$ is differentiable except for a finite number of values of~$s$, and its derivative is given by the first variation formula; see~\cite[Lemma~3.2.4]{BZ}.
Specifically, as long as~$B_x(s)$ is nonempty, we have
\[
L'(\mathcal{C}_x(s)) = \int_{\mathcal{C}_x(s)} \kappa(t) \, dt + \tau_s
\]
for almost every~$s$, where $\kappa$ is the geodesic curvature of the curve~$\mathcal{C}_x(s)$ and $\tau_s$ is the sum of the angular difference of the tangent vectors at the corner points of~$\mathcal{C}_x(s)$.
Denote by~$\omega$ the curvature measure on~$M$.
(In the smooth case, $\omega = K dx$, where $dx$ is the area measure of~$M$.)
By the Gauss--Bonnet formula for piecewise flat metrics and since $B_x(s)$ is a topological ball, we derive
\[
L'(\mathcal{C}_x(s)) = 2 \pi - \omega(B_x(s)).
\]
Integrating this relation twice leads to
\begin{align*}
\area B_x(r) & = \pi r^2 - \int_0^r \int_0^\rho \omega(B_x(s)) \, ds \, d\rho \\
 & = \pi r^2 - \int_0^r (r-\rho) \, \omega(B_x(\rho)) \, d\rho.
\end{align*}
\end{proof}

\forgotten

Putting everything together, we obtain the following area lower bound.

\begin{proposition}
\label{prop:end}
Let $r\in(0,\frac{1}{2}\sys(M))$ where $M$ is a nonpositively curved
surface.  Then there exists a ball~$B(r)$ of radius~$r$ with
\[
\area B(r) \geq \pi r^2 - \frac{2\pi \chi(M)}{\area(M)} \cdot
\frac{\pi}{12} r^4.
\]
\end{proposition}

\begin{proof}
By Proposition~\ref{prop:B}, there exists~$x_0 \in M$ such that
\[
G_r(x_0) \leq \frac{2\pi \chi(M)}{\area(M)} \cdot \frac{\pi r^4}{12}.
\]
By Proposition~\ref{prop:theta}, it follows that
\[
\area B_{x_0}(r) \geq  \pi r^2 - \frac{2\pi \chi(M)}{\area(M)} \cdot \frac{\pi}{12} r^4.
\]
\end{proof}

We can conclude the proof of Theorem~\ref{theo:main} as follows.

\begin{proof}[Proof of Theorem~$\ref{theo:main}$]
We apply Proposition~\ref{prop:end} with $r=\frac{1}{2} \sys(M)$.  
Since
\begin{equation}
\area(M) \geq \area B(\tfrac{1}{2} \sys(M)),
\end{equation}
we obtain
\[
\area(M) \geq \left( \frac{\pi}{4} - \frac{\pi^2 \chi(M)}{96} \cdot
\frac{\sys(M)^2}{\area(M)} \right) \sys(M)^2.
\]
Expressed in terms of the systolic area, this inequality takes the following form
\[
\sigma(M) \geq \lambda + \frac{\mu}{\sigma(M)},
\]
where $\lambda = \frac{\pi}{4}$ and $\mu = - \frac{\pi^2 \chi(M)}{96} \geq 0$.
That is, $\sigma(M)^2 - \lambda \, \sigma(M) - \mu \geq 0$.
Hence,
\[
\sigma(M) \geq \frac{\lambda + \sqrt{\lambda^2 + 4 \mu}}{2}.
\]
It follows that
\begin{equation}
\label{eq:sqrt}
\sigma(M) \geq \frac{\pi}{8} \left( 1 + \sqrt{1-\tfrac{2}{3} \chi(M)} \right)
\end{equation}
for $\chi(M) \leq -1$.  Hence
$\sigma(M)\geq\frac{\pi}{8}\left(1+\sqrt{\tfrac{5}{3}}\right)\approx0.899$
and the surface~$M$ is Loewner in this case.  For the Klein bottle
(where $\chi=0$), the minimal value of the systolic area over
nonpositively curved metrics is attained by a square flat metric and
is equal to~$1$.  (Alternatively, the minimal value of the systolic
area of a Riemannian Klein bottle is equal to~$\frac{2 \sqrt{2}}{\pi}
\approx 0.9$; see~\cite{bav}.)  Hence, every closed nonpositively
curved surface is Loewner.
\end{proof}

\section{Corollaries}

\forget
\begin{remark}
Our proof shows that every closed nonpositively curved surface which is not a torus has systolic area at least 
\[
\frac{\pi}{8} \left( 1 + \sqrt{\tfrac{5}{3}} \right) \approx 0.899.
\]
It also shows that every closed nonpositively curved surface of genus~\mbox{$g \geq 2$} has systolic area at least
\[
\frac{\pi}{8} \left( 1 + \sqrt{\tfrac{4g-1}{3}} \right) \geq \frac{\pi}{8} \left( 1 + \sqrt{\tfrac{7}{3}} \right) \approx 0.992.
\]
It follows that nonpositively curved surfaces of genus~$g \geq 4$ have a systolic area greater than the minimal systolic area of a nonpositively curved surface of genus~$2$, which is equal to $3(\sqrt{2}-1) \approx 1.242$; see~\cite{KS06pams}.

It would be interesting to know if this still holds true in genus~$3$, showing a monotonicity of the minimal systolic area in terms of the genus for low genera.
Note that the best value of the systolic area we know of in genus~$3$ is given by a hyperbolic metric (see~\cite{sch}) and is equal to 
\[
\frac{2\pi}{\arcsinh(2+\sqrt{3})^2} \approx 1.528.
\]
\end{remark}
\forgotten

In~\cite{KS06pams} and~\cite{KS15}, we computed the least value of the systolic area over all nonpositively curved metrics on the genus~$2$ surface~$\Sigma_2$ and Dyck's surface~$3\RP^2$.
For the genus~$3$ surface, Calabi~\cite{cal} presented a CAT$(0)$ piecewise flat metric with the lowest systolic area we know of.
See Remark~\ref{rem} below.
Combined with the systolic inequality~\eqref{eq:sqrt}, this leads to the following result. 


\begin{corollary} \label{coro:table}
We obtain the following data, where the third column represents the minimal value of the systolic area over nonpositively curved metrics.
\[
\def\arraystretch{2.3}
\begin{array}{|c|c|c|c|}
\hline
 M & \chi(M) & \sigma_{\leq 0}(M) & source \\
 \hline
 \hline
\T^2 & 0 & \frac{\sqrt{3}}{2} \approx 0.866 & \text{\cite{pu}} \\
\hline
\K^2=2\RP^2 & 0 & 1 & \textrm{obvious} \\
\hline
3\RP^2 & -1 & 1 + \frac{(169-38 \sqrt{19})^{\frac{1}{2}}}{12} \approx 1.152 & \text{\cite{KS15}} \\
\hline
\Sigma_2 & -2 & 3(\sqrt{2}-1) \approx 1.242 &\text{\cite{KS06pams}} \\
\hline
\Sigma_3 & -4 & < \frac{7 \sqrt{3}}{8} \approx 1.515 & \text{\cite{cal}} \\
\hline
\Sigma_{g \geq 3} & 2-2g & \geq \frac{\pi}{8} \left( 1 + \sqrt{\tfrac{4g-1}{3}} \right) \geq 1.144 & \eqref{eq:sqrt} \\
\hline
\Sigma_{g \geq 4} & 2-2g & \geq \frac{\pi}{8} \left( 1 + \sqrt{\tfrac{4g-1}{3}} \right) \geq 1.270 & \eqref{eq:sqrt} \\
\hline
n \RP^2 \text{(with $n \geq 4$)} & 2-n & \geq \frac{\pi}{8} \left( 1 + \sqrt{\tfrac{2n-1}{3}} \right) \geq 0.992 & \eqref{eq:sqrt} \\
\hline
n \RP^2 \text{(with $n \geq 7$)} & 2-n & \geq \frac{\pi}{8} \left( 1 + \sqrt{\tfrac{2n-1}{3}} \right) \geq 1.210 & \eqref{eq:sqrt} \\
\hline
\end{array}
\]
\end{corollary}

We obtain the following consequences.

\begin{corollary}  \label{coro}
Every closed nonpositively curved surface~$M$ other than a torus has a systolic area at least
\[
 \frac{\pi}{8} \left( 1 + \sqrt{\tfrac{7}{3}} \right) \approx 0.992.
\]
\end{corollary}

\begin{corollary}
Every orientable nonpositively curved surface~$\Sigma_g$ of genus
\mbox{$g \geq 2$} has a systolic area at least
\[
\frac{\pi}{8} \left( 1 + \sqrt{\tfrac{11}{3}} \right) \approx 1.144.
\]
\end{corollary}

\begin{definition}
A \emph{Loewner disk} of a closed nonsimply connected surface~$M$ with a Riemannian metric is a disk of radius~$\frac{1}{2} \sys(M)$ and area at least~$\frac{\sqrt{3}}{2} \sys(M)^2$.
\end{definition}

A closed hyperbolic surface of sufficiently small systole does not
contain any Loewner disk.  All disks of radius half the systole have
area close to that of a Euclidean disk of the same radius.  However,
the existence of Loewner disks is guaranteed in the following cases.

\begin{corollary}
  Let $\Sigma_g$ be a nonpositively curved surface of genus~$g\geq2$.
\begin{enumerate}
  \item
Let $C=\frac{\pi^2}{12(2 \sqrt{3}-\pi)}\approx2.5502$. 
If $\sigma(\Sigma_g) \leq C(g-1)$ then $\Sigma_g$ contains a Loewner disk.
\item
If $\Sigma_g$ is systolically extremal then it contains a Loewner
disk.
\end{enumerate}
\end{corollary}

\begin{proof}
The first part of the corollary is immediate from Proposition~\ref{prop:end} with $r=\frac{1}{2} \sys(\Sigma_g)$.

For the second part, consider a rectangle~$R=[0,2g-1] \times [0,1]$ in
the plane.  Subdivide each edge of the rectangle into unit length
segments.  The rectangle~$R$ can be seen as a $4g$-gon with unit sides
labeled circularly $a_1,\cdots,a_{2g},\bar{a}_1,\cdots,\bar{a}_{2g}$.
Identifying every side~$a_i$ of this $4g$-gon with its corresponding
side~$\bar{a}_i$ by a plane translation, we obtain a translation
surface which is a piecewise flat genus~$g$ surface~$\Sigma_g$ with a
single conical singularity (which can be approximated by a
nonpositively curved metric); see~\cite[p.~21]{FM}.
It has area~$2g-1$ and systole~$1$.  Therefore, every nonpositively
curved systolically extremal surface of genus~$g\geq3$ has systolic
area at most~$2g-1$ and contains a Loewner disk by the first part of
the corollary.  For $g=2$, we use the minimal value of the systolic
area given in Corollary~\ref{coro:table}.
\end{proof}

\begin{remark}
Similarly, one can show that every nonpositively curved systolically extremal surface of negative Euler characteristic contains a Loewner disk.
\forget
We can glue the edges according to a_1 a_2 ... a_g a_1^{-1} a_2^{-1} in the orientable (genus g surface), or according to c_1 c_1 ... c_g c_g in the non-orientable cases (g cross-caps).

This is not enough to conclude for 3RP^2, but we know the value of the optimal systolic area anyway, of for 4RP^2. In the latter case, we can take a unit square, identify the sides to obtain a Klein bottle and slit it along a horizontal or vertical segment of length 1/2 in the middle of the square. Now, we glue two copies of these punctured Klein bottle to form 4RP^2.
\forgotten
\end{remark}

\forget

Normalizing to unit systole, we will say that a \emph{Loewner disk} is
a $\frac12$-disk of area at least $\frac{\sqrt3}{2}$.

\begin{corollary}
\label{c44}
Let $\Sigma_g$ be a nonpositively curved orientable surface of unit
systole.  Let $C=\frac{\pi^2}{24\sqrt3-12\pi}\approx2.5502$.  If
$\area(\Sigma_g)\leq C(g-1)$ then $\Sigma_g$ contains a Loewner disk.
\end{corollary}

\begin{proof}
This is immediate from Proposition \ref{prop:end} with $r=\frac12$.
\end{proof}

For example, a systolically extremal nonpositively curved surface of
genus~$3$ must contain a Loewner disk because its area satisfies the
hypothesis of Corollary~\ref{c44} by estimate \eqref{e41} below.  One
can show that, similarly, every nonpositively curved, systolically
extremal surface of negative Euler characteristic must contain a
Loewner disk.

\forgotten

\begin{remark} \label{rem}
Observe that nonpositively curved surfaces of genus \mbox{$g \geq 4$} have a systolic area greater than the minimal systolic area of a nonpositively curved surface of genus~$2$, which is equal to $3(\sqrt{2}-1) \approx 1.242$; see~\cite{KS06pams}.
It would be interesting to know if this still holds true in genus~$3$, showing a monotonicity of the minimal systolic area in terms of the genus for low genera.
Observe that the best value of the systolic area we know of in genus~$3$ is given by a CAT$(0)$ piecewise flat surface in the conformal class of the Klein quartic described by Calabi~\cite{cal} and is equal to
\begin{equation}
\label{e41}
\frac{7 \sqrt{3}}{8} \approx 1.515.
\end{equation}
This metric is a critical with respect to some metric variations (see~\cite{sab}) and might be extremal among all metrics without any curvature assumption.
It is not surprising that such a metric is piecewise flat since optimal CAT$(0)$ metrics are flat with finitely many singularities in every genus by~\cite{Ka21}.
Note that Calabi's surface has only a slightly better systolic area than the one given by the triangle hyperbolic surface $(2,3,12)$ of the same genus described by Schmutz~\cite{sch} (and conjectured extremal among hyperbolic metrics), which is equal to 
\[
\frac{2\pi}{\arcsinh(2+\sqrt{3})^2} \approx 1.528.
\]
It would be interesting to find the optimal CAT$(0)$ metric in the
conformal class of the surface described by Schmutz (or at least find
a good approximation by piecewise flat metrics) to see if it has a
lower systolic area than the one of Calabi's surface.



Similarly, nonpositively curved surfaces~$n\RP^2$ with $n \geq7$ have a systolic area greater than the minimal systolic area of a nonpositively curved Dyck's surface~$3\RP^2$, which is equal to 
\[
1 + \frac{(169-38 \sqrt{19})^{\frac{1}{2}}}{12} \approx 1.152
\]
(see~\cite{KS15}) and the question is open for $n=4,5,6$.
\end{remark}

\section*{Acknowledgments}

Mikhail Katz was supported by the BSF grant 2020124 and the ISF grant
743/22.
St\'ephane Sabourau was supported by the ANR project Min-Max (ANR-19-CE40-0014).


\begin{thebibliography}{99}

\bibitem{bav}
Bavard, C.\, In\'egalit\'e isosystolique pour la bouteille de Klein.
\emph{Math. Ann.} 274 (1986), 439--441. 

\bibitem{besse}
Besse, A. L.\, Manifolds all of whose geodesics are closed.
\emph{Ergeb. Math. Grenzgeb.} 93, Springer-Verlag, 1978. 


\bibitem{cal}
Calabi, E.\, Extremal isosystolic metrics for compact surfaces. 
Actes de la Table Ronde de G\'eom\'etrie Diff\'erentielle, \emph{S\'emin. Congr.} 1 (1996), Soc. Math. France, 146--166.

\bibitem{CE} Cheeger, J.; Ebin, D.\, Comparison theorems in Riemannian
  geometry.  \emph{North-Holland Math. Library}, Vol. 9, North-Holland
  Publishing, 1975

\bibitem{FM}
Farb, B.; Margalit, D.\, A primer on mapping class groups.
\emph{Princeton Math. Series}, Vol. 49, Princeton University Press, 2011.

\bibitem{Go24} Goodwillie, T.; Hebda, J.; Katz, M.\, Extending
  Gromov's optimal systolic inequality.  \emph{Journal of Geometry}
  \textbf{114} (2023), article 23.
  \url{https://doi.org/10.1007/s00022-023-00685-3}
  

\bibitem{KS05}
Katz, M.; Sabourau, S.\, Entropy of systolically extremal surfaces and asymptotic bounds.
\emph{Ergodic Theory Dynam. Systems} 25 (2005), no.~4, 1209--1220.

\bibitem{KS06pams}
Katz, M.; Sabourau, S.\, Hyperelliptic surfaces are Loewner.
\emph{Proc. Amer. Math. Soc.} 134 (2006), no.~4, 1189--1195.

\bibitem{KS15}
Katz, M.; Sabourau, S.\, Dyck's surfaces, systoles, and capacities.
\emph{Trans. Amer. Math. Soc.} 367 (2015), no.~6, 4483--4504. 

\bibitem{Ka21} Katz, M.; Sabourau, S.\, Systolically extremal nonpositively curved surfaces are flat with finitely many singularities.  \emph{J. Topol. Anal.} 13 (2021), no.~2, 319--347.



  

\bibitem{Ka24a} Katz, M.; Sabourau, S.\, Logarithmic systolic growth
  for hyperbolic surfaces in every genus.  \emph{Proceedings of the
  American Mathematical Society} (2024).
  \url{https://arxiv.org/abs/2407.02041}


\bibitem{LS}
Li, Q.; Su, W.\, Every closed surface of genus at least~$18$ is Loewner.
See arXiv:2401.00720v2.

\bibitem{pu}
Pu, P.M.\, Some inequalities in certain nonorientable Riemannian manifolds.
\emph{Pacific J. Math.} 2 (1952), 55--71.

\bibitem{sab}
Sabourau, S.\, Isosystolic genus three surfaces critical for slow metric variations.
\emph{Geom. Topol.} 15 (2011), no.~3, 1477--1508.
      
      
\bibitem{sch}
Schmutz, P.\, Riemann surfaces with shortest geodesic of maximal length.
\emph{Geom. Funct. Anal.} 3 (1993), no.~6, 564--631 . 


\end{thebibliography}
\end{document}